\documentclass[12pt]{amsart}
\usepackage{graphicx}

\usepackage{amssymb} 

\usepackage[color,matrix,arrow]{xy}
\usepackage{tikz}
\usetikzlibrary{decorations.markings,arrows,automata,arrows,backgrounds,snakes}
\usepackage{graphicx}

\usepackage{subcaption} 

\usepackage{hyperref} 

\hypersetup{bookmarks=true,
	unicode=true,
	colorlinks=true,
	citecolor=black,
	linkcolor=black,
	urlcolor=black,
	plainpages=false,
	pdfpagelabels=true}

\usepackage{tikz-cd}

\usepackage{pgf}
\usepackage[numbers]{natbib}

 \usepackage{url}	
 \allowdisplaybreaks 

\newtheorem{teo}{Theorem}[subsection]
\newtheorem{theorem}[teo]{Theorem}
\newtheorem*{thm*}{Theorem}

\newtheorem{corollary}[teo]{Corollary}

\newtheorem{lema}[teo]{Lemma}
\newtheorem{lemma}[teo]{Lemma}
\newtheorem{prop}[teo]{Proposition}

\newtheorem*{claim*}{Claim}

\theoremstyle{definition}
\newtheorem{definition}[teo]{Definition}
\newtheorem{example}[teo]{Example}

\theoremstyle{remark}
\newtheorem{remark}{Remark}

\numberwithin{figure}{section}

\newcommand{\co}{\colon}
\newcommand{\R}{\mathbb{R}}

\newcommand{\crit}{\mathrm{crit}}

\newcommand{\sd}{\mathrm{sd}}

\begin{document}
\title[Morse Theory on posets]{Fundamental Theorems of Morse Theory on posets}
\thanks{The first and the fifth authors were partially supported by MINECO Spain Research Project MTM2015--65397--P and Junta de Andaluc\'ia Research Groups FQM--326 and FQM--189. The second author was partially supported by MINECO-FEDER research project MTM2016--78647--P. The third author was partially supported by Ministerio de Ciencia, Innovaci\'on y Universidades,  grant FPU17/03443. The second and third authors were partially supported by Xunta de Galicia ED431C 2019/10 with FEDER funds.}
\author[Fern\'{a}ndez-Ternero, Mac\'{i}as-Virg\'os, Mosquera-Lois, Vilches]{D. Fern\'{a}ndez-Ternero$^{(1)}$, E. Mac\'{i}as-Virg\'os$^{(2)}$, \\
D. Mosquera-Lois$^{(2)}$, N.A. Scoville$^{(3)}$, J.A. Vilches$^{(1)}$}
\address{$^{(1)}$ Departamento de Geometr\'ia y Topolog\'ia, Universidad de Sevilla, Spain.}
\address{$^{(2)}$ Instituto de Matem\'aticas, Universidade de Santiago de Compostela, Spain.}
\address{$^{(3)}$ Department of Mathematics and Computer Science, Ursinus College, Collegeville PA 19426 U.S.A.}
\email{desamfer@us.es, quique.macias@usc.es, david.mosquera.lois@usc.es, nscoville@ursinus.edu, vilches@us.es}

\begin{abstract} We prove a version of the fundamental theorems of Morse Theory in the setting of finite spaces or partially ordered sets. By using these results we extend Forman's discrete Morse theory to more general cell complexes and derive the Morse-Pitcher inequalities in the context of finite spaces. 
\end{abstract}


\maketitle

\tableofcontents

\section{Introduction}\label{INTRO}

Morse Theory was originally introduced as a tool for the study of variational problems on manifolds by relating the homology of a space with some critical objects arising after defining a map and considering its induced dynamics \cite{Morse}. Since its introduction, Morse Theory has been an active field of research and connections with many different areas of Mathematics have been found. That interaction has led to several adaptations of Morse Theory to different contexts: PL versions by Banchoff \cite{Banchoff,Banchoff2} and by Bestvina and Brady \cite{Bestvina} and a purely combinatorial approach by Forman \cite{Forman,Forman2}. Nowadays, not only pure mathematics benefit from that interaction, but also applied mathematics \cite{Ghristbook} due to the importance of discrete settings.

Roughly speaking, Morse Theory addresses the study of the topology (homotopy or homology, originally) of a space by breaking it into ``elementary'' pieces. That is achieved by the so called Fundamental or Structural Theorems of Morse Theory, which assert that the object of study (for example a smooth manifold or a simplicial complex) has the homotopy type of a CW-complex with a given cell structure determined by the criticality of a Morse function defined on it \cite{Forman2,Milnor}.

Recent works have shown that it is possible to approach important problems regarding posets by using topological methods. See for example Barmak and Minian's work on realizing groups as the automorphism groups of certain posets \cite{BARMAK_posets} or Stong's work on groups dealing with the way the homotopy type of the poset of  non-trivial $p$-subgroups  ordered by inclusion determines algebraic properties of the group  \cite{Stong_groups}. Therefore, it makes sense to study the homotopy type of finite spaces by means of some adapted version of Morse theory.

The aim of this work is to develop an extension of Morse Theory to finite spaces introduced by Minian \cite{Minian} in order to prove the Fundamental Theorems of Morse Theory in this setting. Being more precise,
we study the evolution of level subposets with or without critical values. Moreover, some of their consequences are exploited, for instance providing an alternative proof of Forman's decomposition theorem \cite[Corollary 3.5]{Forman2}, an  extension of Discrete Morse Theory to more general cell complexes and the Morse-Pitcher inequalities.

The organization of the paper is as follows.  In Section \ref{sec:preliminaries} we recall some definitions and standard results about posets or finite topological spaces. Section \ref{sec:morse_theory_posets} is devoted to the study of discrete Morse Theory in the context of posets. In Section \ref{sec:fundamental_theorems_and_consequences} we prove the Fundamental Theorems of Morse Theory in this setting. Finally, in Section \ref{sec:consequences} we show some of their consequences, such us extending original Forman's result regarding the homotopy type of a regular CW-complex with a Morse function defined on it or obtaining the Morse inequalities. Moreover, we study a method to reduce the criticality of a Morse function defined on a poset.


\section{Finite Spaces, posets, and simplicial complexes} \label{sec:preliminaries}

This section is devoted to introduce the objects we will work with. In particular we are interested in two kinds of posets, two-wide posets
and down-wide posets (the first of which is due to Bloch \cite{Bloch}),
for which we will establish the main results of this work. Most of the material presented is well established in the literature, for further details the reader is referred to \cite{Barmak_book,BarMin12,Bloch,Farmer,Minian,Wachs}.
All posets will be assumed to be finite and by finite space we will mean $T_0$-space.
\subsection{Preliminaries}

It is well known  that finite posets and finite $T_0$-spaces are in bijective correspondence.  If $(X, \leq)$ is a poset, a topology on $X$ is defined by taking as a basis the open sets
 $$U_x:=\{y\in X: y\leq x\}$$ for each $x\in X$.   On the other hand, if $X$ is a finite $T_0$-space, for any $x\in X$, define the {\it minimal open set} $U_x$ as the intersection of all open sets containing $x$.  Then $X$ may be given a poset structure by defining $y\leq x$ if and only if $U_y\subseteq U_x$. It is easy to see that  these correspondences are mutual inverses of each other. Moreover a map between posets $f\co X\to Y$ is order preserving if and only if it is continuous when considered as a map between the associated finite spaces.   As a consequence of the correspondence between posets and finite $T_0$-spaces, we will use both notions interchangeably.

We need to introduce some basic notions and results.

\begin{definition}
	A {\it chain} in the poset $X$ is a subset $C\subseteq X$ such that if $x,y\in C$, then either $x\leq y$ or $y\leq x$. The {\it height} of $X$ is the maximum length of the chains in $X$, where a chain $x_0<x_1<\ldots <x_n$ has length $n$. The height $h(x)$ of an element $x\in X$ is the height of $U_x$ with the induced order.
\end{definition}

\begin{definition}
	 A poset $X$ is said to be {\it homogeneous} of degree $n$ if all maximal chains in $X$ have length $n$. A poset is {\it graded} if $U_x$ is homogeneous for every $x\in X$. In that case, the degree of $x$, denoted by $\deg(x)$, is its height.
\end{definition}

We will denote both the height and degree of an element by superscripts, that is, $x^{(p)}$.

Let $X$ be a finite poset, $x,y\in X$. If $x< y$ and there is no $z\in X$ such that $x< z< y$, we write $x\prec y$.


For $x\in X$ we also define $\widehat{U}_x:=\{w\in X\colon w<x\}$ as well as $F_x:=\{y\in X : y\geq x\}$ and $\widehat{F}_x:=\{y\in X : y>x\}$.

\subsection{Beat points and $\gamma-$points} Due to the correspondence between posets and finite topological spaces we can study the homotopy type of a poset.

\begin{definition}
	A point $x\in X$ is a {\it beat point} if either $\widehat{U}_x$ has a maximum ({\it down beat point}) or $\widehat{F}_x$ has a minimum ({\it up beat point}). Notice that in both, $\widehat{U}_x$ and $\widehat{F}_x$  are contractible.
\end{definition}

The next proposition states that removing beat points from a poset does not change its homotopy type.

\begin{prop}\cite[Proposition 1.3.4]{Barmak_book} \label{prop:removing_beat_point_homot_equiv} Let $x\in X$ be a beat point.  Then $X-\{x\}$ is a strong deformation retract of $X$.
\end{prop}


There is a weaker notion of beat point which we recall now:

\begin{definition}{\cite{Barmak_Minian},\cite[Definition 6.2.1]{Barmak_book}}
	The point $x\in X$ is a {\it $\gamma-$point} if $\widehat{C}_x=(U_x\cup F_x)-\{x\}$ is homotopically trivial.
\end{definition}

\subsection{McCord Theory of weak equivalences}

We now recall McCord functors between posets and simplicial complexes. 
 Given a poset $X$, we define its {\it order complex} $\mathcal{K}(X)$ as the simplicial complex whose $k$-simplices are the non-empty chains of $X$ of length $k$. Furthermore, given an order preserving map $f\co X\to Y$ between posets, we define a simplicial map $\mathcal{K}(f)\co \mathcal{K}(X) \to \mathcal{K}(Y)$ by $\mathcal{K}(f)(x)=f(x)$.

  Conversely, if $K$ is a simplicial complex, we define the face poset of $K$, denoted $\Delta(K)$, as the poset of simplices of $K$ ordered by inclusion. Given a simplicial map $\phi\co K \to L $ we define the order preserving map $\Delta(\phi)\co \Delta(K) \to \Delta(L)$ by $\Delta(\phi)(\sigma)=\phi(\sigma)$ for each simplex $\sigma$ of $K$.

  The face poset functor is defined analogously for CW-complexes. That is, given a CW-complex $K$, $\Delta(K)$ is the poset of cells of $K$ ordered by inclusion. Given a cellular map $\phi\co K \to L $ we define the order preserving map $\Delta(\phi)\co \Delta(K) \to \Delta(L)$ by $\Delta(\phi)(\sigma)=\phi(\sigma)$ for each cell $\sigma$ of $K$. Note that for a simplicial complex $K$, $\mathcal{K}\Delta(K)$ is $\sd(K)$, the first barycentric subdivision of $K$. For details and a proof of the result below see \cite{Barmak_book}:
  
  \begin{theorem}\label{thm:mccords_thms}
  	The following statements hold:
  	\begin{enumerate}
  		\item Let $X$ be a finite $T_0$-space. Then there is a map $\mu_X\co |\mathcal{K}(X)|\to X$ which is a weak homotopy equivalence.
  		\item Let $K$ be a simplicial complex. Then there is a map $\mu_K\co |K|\to \Delta(K)$ which is a weak homotopy equivalence.
  	\end{enumerate}
  \end{theorem}

In the proof of Theorem \ref{thm:mccords_thms} played a central role McCord's Theorem, which we now present. Let $X$ be a topological space. An open cover $\mathcal{U}$ of $X$ is called a {\it basis like open cover} for $X$ if $\mathcal{U}$ is a basis for a topology in the underlying set of $X$. Note that given a finite space $X$, the minimal basis $\{U_x\}_{x\in X}$ is a basis like open cover for $X$.

We reproduce the statement of McCord's Theorem:

\begin{theorem}\cite[Theorem 6]{McCord}\label{thm:McCords_thm}
	Let $f\colon X\to Y$ be a continuous map between topological spaces. Suppose there is a basis like open cover $\mathcal{U}$ of $Y$ such that for every $U\in \mathcal{U}$, the restriction:
	$$f_{\vert f^{-1}(U)}\colon f^{-1}(U)\to U$$
	is a weak homotopy equivalence. Then $f\colon X\to Y$ is a weak homotopy equivalence.
\end{theorem}

\subsection{Two-wide posets}

In this subsection we introduce a class of posets important for the later development of Morse Theory.

\begin{definition}[\cite{Bloch}]
	A poset $X$ is {\it two-wide}  if for any $x,z,y$ such that $x\prec z \prec y$, there is some $z'\in X$ such that $z'\neq z$ and $x\prec z' \prec y$.	
\end{definition}

\begin{remark}
	A poset $X$ is two-wide if and only if it satisfies the following condition: for any pair of elements $x,y\in X$ such that $x<y$ and $x\nprec y$, $\#\{z\colon x\leq z\leq y\}\geq 4$.
\end{remark}

\begin{lema}{\cite[Lemma 4.1, p. 168]{Lundell}}
	Given a regular CW-complex $K$, its face poset $\Delta(K)$ is two-wide.
\end{lema}

\subsection{Down-wide posets}
In this subsection we introduce the new concept of down-wide poset, which will play an important role in the development of Morse Theory in the context of finite posets. 

\begin{definition}
	Given a poset $X$ and $x\in X$, we define the {\em down-incidence number} of $x$ as the cardinality of the set $\partial(x)=\{y \in X\colon y\prec x\}$. The poset $X$ is {\em down-wide} if $\#\partial(x)\geq2$ for every for every non-minimal $x$ in $X$. 
\end{definition}

Obrserve that down-wide posets do not have down-beat points. It is easy to check the following result:

\begin{lema}\label{lem:face_poset_is_cellular}
	For any regular CW-complex $K$, its face poset $\Delta(K)$ is down-wide.
\end{lema}

Therefore, all posets coming from simplicial complexes by the McCord functor $\Delta$ are down-wide. However, not every down-wide poset is the face poset of some simplicial complex.

\begin{example}
	The poset $X$ pictured in Figure \ref{fig:cellular_poset_which_is_face_poset} is down-wide.
	\begin{figure}[htbp]
	\centering
	\includegraphics[scale=0.5]{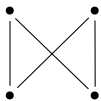}
	\caption{A down-wide poset which is not the face poset of any simplicial complex.}
	\label{fig:cellular_poset_which_is_face_poset}
\end{figure}
However, it is not the face poset of any simplicial complex $K$. Otherwise $K$ would have two $0$-simplices and two $1$-simplices.
\end{example}

\begin{definition}[\cite{Minian}]
A poset $X$ is a {\it model} of a CW-complex $K$ if the geometric realization of $\mathcal{K}(X)$ is homotopy equivalent to $K$.
\end{definition}

Observe that a subposet $Y$ of a poset $X$ is an open subposet (seeing both posets as finite spaces) if whenever $x\in Y$ and $y\leq x$, then $y\in Y$. Therefore, an open subposet of a down-wide poset is down-wide.

Observe that the properties of being two-wide and down-wide do not imply each other.

\subsection{(Homologically) admissible posets}

We recall the notion of (homologically) admissible posets introduced by Minian \cite{Minian}. We denote by $\mathcal{H}(X)$ the Hasse diagram associated to the poset $X$. 

\begin{definition}[\cite{Minian}]
	Let $X$ be a poset. An edge $(w,x)\in \mathcal{H}(X)$ is admissible if $\widehat{U}_x-\{w\}$ is homotopically trivial. A poset is {\it admissible} if all its edges are admissible.
\end{definition}

In order to check if an edge $(w,x)\in \mathcal{H}(X)$ is admissible, one needs to compute the higher homotopy groups of $\widehat{U}_x-\{w\}$, which is a difficult problem. That is why we introduce the following weaker notion:

\begin{definition}
	Let $X$ be a poset. An edge $(w,x)\in \mathcal{H}(X)$ is $1$-admissible if $\widehat{U}_x-\{w\}$ is simply connected. A poset is $1$-{\it admissible} if all its edges are $1$-admissible.
\end{definition}

We also recall the notion of homologically admissible poset, which is weaker than admissible poset. Recall that given a poset $X$, since it is also a finite topological space, its singular homology is defined. 

\begin{definition}[\cite{Minian}]
	Let $X$ be a poset. An edge $(w,x)\in \mathcal{H}(X)$ is homologically admissible if $\widehat{U}_x-\{w\}$ is acyclic. A poset is {\it homologically admissible} if all its edges are homologically admissible.
\end{definition}

\begin{remark}
	Observe that by the Hurewicz Theorem, an edge $(w,x)\in \mathcal{H}(X)$ is admissible if and only if it is $1$-admissible and homologically admissible.
\end{remark}

\begin{remark}
	The face posets of regular CW-complexes are admissible, and therefore homologically admissible \cite[Remark 2.13]{Minian}. However, not every (homologically) admissible poset is the face poset of a regular CW-complex as \cite[Example 2.4]{Minian} illustrates.
\end{remark}

We present below an important class of examples of homologically admissible posets. For a detailed exposition the reader is referred to \cite{Bjorner,Minian,Negami}).

\begin{definition}[\cite{Minian}]
	A simplicial complex $K$ is a {\it closed homology manifold} of dimension $n$ if the link of every simplex has the homology of $\mathbb{S}^{n-k-1}$, where $k$ is the dimension of the simplex. 	A poset $X$ is a {\it finite closed homology manifold} if its order complex $\mathcal{K}(X)$ is a closed homology manifold.
\end{definition}

\begin{lema}{\cite[Theorem 4.6]{Minian}}\label{lem:closed_homology_manifol_is_homologically_admissible}
	If $X$ is a finite closed homology manifold, then it is homologically admissible.
\end{lema}

As a consequence of this result, the theory developed for homologically admissible manifolds can be applied to study the topology of triangulable homology manifolds by means of their order triangulations.

We end the subsection by relating the properties of being (ho\-mo\-lo\-gi\-cally)-admissible with those of being two-wide and down-wide. First, it is easy to check the flowing lemma:

\begin{lemma}\label{lemma:homologiaclly_admissible_is_down-wide}
	Let $X$ be a poset. If $X$ is homologically admissible, then it is down-wide.
\end{lemma}

\begin{remark}
	In Lemma \ref{lemma:homologiaclly_admissible_is_down-wide} it is assumed that the empty set is not acyclic.
\end{remark}

Second, the cellular homology developed by Minian in \cite{Minian} allows to adapt the ideas of \cite[Theorem 1.2 (iii)]{Forman2} so the proposition below follows easily:

\begin{prop}\label{prop:homologically_admissible_is_two_wide}
	Let $X$ be a poset. If $X$ is homologically admissible, then it is two-wide.
\end{prop}
%

\section[Morse Theory on posets]{Morse Theory on posets} \label{sec:morse_theory_posets}

\subsection{Definition of Morse functions}
We recall the definition of Morse function for posets introduced by Minian \cite{Minian}. It is an adaptation of Forman's theory \cite{Forman,Forman2} to the context of posets.

\begin{definition} Let $X$ be a finite poset.  A  {\it Morse function} on $X$ is a map $f\colon X \to \R$  such that, for every $x\in X$, we have
	$$
	\#\{y\in X\colon x\prec y \text{ and } f(x)\geq f(y)\}\leq 1
	$$
	and
	$$
	\#\{w\in X \colon w\prec x \text{ and } f(w)\geq f(x)\}\leq 1.
	$$
\end{definition}

\begin{definition}
		If $f$ is a Morse function, the point $x\in X$ is said to be {\it critical} if
	$$
	\#\{y\in X \colon x\prec y \text{ and } f(x)\geq f(y)\}=0
	$$
	and
	$$
	\#\{w\in X \colon w\prec x \text{ and } f(w)\geq f(x)\}=0.
	$$
\end{definition}

The set of critical points is denoted by $\crit{f}$. The images of the critical points are called {\it critical values}. The points (values) which are not critical are said to be {\it regular points} ({\it regular values}).

\subsection{Technical Lemmas for Morse functions}
We begin by stating a result that plays the role of two important theorems developed by Forman in the simplicial setting \cite[Theorems 1.2 and 1.3]{Forman2}. In fact, the proof of the following Key Lemma is much simpler in this context than in the simplicial setting so it is left for the reader.

\begin{lema}[Key Lemma]\label{lemma:key_lemma}
	Suppose that $X$ is a finite two-wide poset and there are two elements $w<y$ such that $w\nprec y$. Then there are elements $x,\tilde{x}$ such that:
	\begin{itemize}
		\item $x\neq \tilde{x}$
		\item $w\prec x < y$ and $w \prec \tilde{x} < y$.
	\end{itemize}
\end{lema}

\begin{remark}
	Observe that Lemma \ref{lemma:key_lemma} does not hold in general for finite spaces. As an example, consider the poset of Figure \ref{fig:counterexample_exclussion_lemma_h_reular_posets} taking the points labeled as $z$ and $y$.
\end{remark}


%

\begin{definition}
	Given a poset $X$, a Morse function $f\colon X\to \R$ is said to satisfy the {\it Exclusion condition} if for every regular point $x \in X$, exactly one of the following conditions holds:
	\begin{enumerate}
		\item There exists exactly one $y \in X$, $x\prec y$, such that $f(x)\geq f(y)$,
		\item There exists exactly one $w \in X$, $w\prec x$, such that $f(w) \geq f(x)$.
	\end{enumerate}
\end{definition}

The following result plays the role, in the context of finite spaces, of the Exclusion Lemma \cite[Lemma 2.5]{Forman2}.

\begin{lema}[Exclusion Lemma]\label{lma:exclusion:finite}
	Let $X$ be a finite two-wide poset and $f\colon X\to \R$ a Morse function on $X$. Then $f\colon X\to \R$ satisfies the Exclusion condition.
\end{lema}

\begin{proof}
	First of all, since $x$ is not critical, then at least one of the conditions holds. We will see that the conditions are mutually exclusive. So, assume both conditions hold and we will arrive to a contradiction. By Condition $2$, $x$ is not a minimal element. By Lemma \ref{lemma:key_lemma} there exists $x'$ such that: $x'\neq x$ and  $y>x'\succ w$. Since $x'<y$, Condition $2$ of the definition of Morse function applied to $y$ gives $f(x')<f(y)$. Hence $f(x')<f(x)$. By Condition $1$ applied to $w$, it holds that $f(w)<f(x')$. As a consequence we obtain the following chain of inequalities:
	\begin{equation*}
	f(x)\leq f(w)<f(x')<f(y)\leq f(x).
	\end{equation*}
	So $f(x)<f(x)$, which is a contradiction.
\end{proof}

It is interesting to point out that the Exclusion Lemma does not necessarily hold in general for posets which are not two-wide, as the following example shows.

\begin{example}
	Consider the following down-wide model of $\mathbb{S}^3$ (taken from \cite[Fig.2]{Minian}) with the Morse function $f$ represented in Figure \ref{fig:counterexample_exclussion_lemma_h_reular_posets} by the labelling of the points.
	\begin{figure}[htbp]
		\centering
		\includegraphics[scale=0.5]{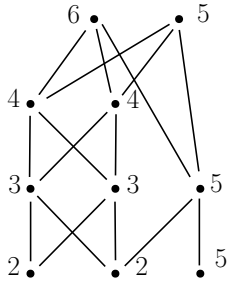}
		\caption{Lemma \ref{lma:exclusion:finite} may not hold for posets which are not two-wide.}
		\label{fig:counterexample_exclussion_lemma_h_reular_posets}
	\end{figure}
Hence Lemma \ref{lma:exclusion:finite} may not hold for arbitrary posets. 
\end{example}

\subsection{Matchings}

We recall the definition of matching introduced to the context of discrete Morse theory by Chari \cite{Chari} and further developed by Minian \cite{Minian}.

\begin{definition}
	A {\it matching}  $\mathcal{M}$ in a poset $X$ is a subset $\mathcal{M} \subseteq X\times X$ such that
	\begin{itemize}
		\item
		$(x, y) \in \mathcal{M}$ implies $x\prec y$;
		\item
		each $x \in X$ belongs to at most one element in $\mathcal{M}$.
	\end{itemize}
\end{definition}

\begin{definition}
	A matching is admissible if each element of the matching is admissible. The notions of homologically and $1$-admissible matchings are defined analogously.
\end{definition}

Let $\mathcal{H}(X)$ be the Hasse diagram of a poset $X$.  If $\mathcal{M}$ is a matching in $X$, write $\mathcal{H}_{\mathcal{M}}(X)$ for the directed graph obtained from $\mathcal{H}(X)$ by reversing the orientations of the edges which are not in $\mathcal{M}$.

\begin{definition}
	  The matching $\mathcal{M}$ is a {\it Morse matching} on $X$ if $\mathcal{H}_{\mathcal{M}}(X)$ is acyclic.  Any point of $\mathcal{H}(X)$ not incident with an edge of $\mathcal{M}$ is called {\it critical}. The set of all critical points of $\mathcal{M}$ is denoted $\crit({\mathcal{M}})$.
\end{definition}

Inspired by the notions of h-regular and cellular posets introduced in \cite{Minian} we introduce the following definition.
\begin{definition}
	A matching $\mathcal{M}$ in $X$ is homology-regular if for every $x^{(p)}\in \crit(\mathcal{M})$, the subspace $\widehat{U}_x$ has the homology of $\mathbb{S}^{p-1}$ where $p$ is the height of $x$. A matching $\mathcal{M}$ in $X$ is homotopy-regular if for every $x^{(p)}\in \crit(\mathcal{M})$, the subspace $\widehat{U}_x$ is a finite model of $\mathbb{S}^{p-1}$ where $p$ is the height of $x$. 
\end{definition}

\begin{example}
	Any matching defined in the poset pictured in Figure \ref{fig:counterexample_exclussion_lemma_h_reular_posets} is homotopy-regular, and therefore homology-regular, since for every $x^{(p)}\in \crit(\mathcal{M})$, the subspace $\widehat{U}_x$ has the homotopy type of a finite model of $\mathbb{S}^{p-1}$.
\end{example}

Minian proved an integration result for matchings which can be stated as follows.

\begin{theorem}\cite[Lemma 3.12]{Minian} \label{thm:given_matching_there_is_function}
	Let $X$ be a finite graded poset and let $\mathcal{M}$ be a a Morse matching on $X$. Then there is a Morse function $f\colon X \to \R$ such that $\crit(f)=\crit({\mathcal{M}})$. Moreover, the function $f\colon X \to \R$ is order preserving, that is, if $x\leq y$, then $f(x)\leq f(y)$.
\end{theorem}

As a consequence of our Exclusion Lemma for Morse functions on two-wide posets (Lemma \ref{lma:exclusion:finite}), we obtain a converse result.

\begin{theorem}\label{thm:morse_function_implies_morse_matching}
	Let $X$ be a finite poset and let $f\co X \to \R$ be a Morse function  satisfying the Exclusion condition. Then there exists an associated Morse matching $\mathcal{M}_f$ with the same critical set. In particular, given a finite two-wide poset $X$ and a Morse function $f\co X \to \R$, there exists an associated Morse matching $\mathcal{M}_f$ with $\crit(f)=\crit({\mathcal{M}})$.
\end{theorem}

\begin{corollary}\label{coro:exclusive_morse_function_implies_exclusive_order_preserving_one}
	Let $X$ be a finite graded poset and let $f\co X \to \R$ be a Morse function  satisfying the Exclusion condition. Then there exists an order preserving Morse function $f'\colon X\to \R$ satisfying the Exclusion condition with the same associate Morse matching that $f$.
\end{corollary}

 Therefore, we can establish a correspondence between Morse matchings and order preserving Morse functions satisfying the Exclusion condition on graded posets. However, the correspondence is not bijective since given a Morse function $f\co X \to \R$, the function $f'\co X \to \R$ given by $f'(x)=2f(x)$ is again Morse and both functions share the same associated matching.
%


\section{Fundamental Theorems} \label{sec:fundamental_theorems_and_consequences}

\subsection{First observations}

We introduce the following notation: given a finite poset $X$ and a  Morse function $f\co X\to \R$, for each $a\in \mathbb{R}$ we denote
$$X_a^f=\bigcup\limits_{f(x)\leq a}{U_x}.$$
Observe that, for each $a\in \mathbb{R}$, the subposet $X_a$ is an open subset of $X$.

We denote by $b_0(X)$ the number of connected components of $X$, which coincides with the number of path-connected components. Given a discrete Morse function on a simplicial complex $f\co K\to \R$, it holds that, as $a\in \R$ increases, new connected components of $K_a$  arise as critical vertices. The following result asserts this phenomenon for down-wide posets.

\begin{prop}\label{prop:path_components_critical_values}
	Let $X$ be a path-connected finite down-wide poset and let $f\colon X \to \mathbb{R}$ be an injective Morse function. Given $a,b\in \R$, $a<b$, if $b_0(X_a)<b_0(X_b)$, then there exists a critical value $c\in (a,b]$ such that $b_0(X_c)=b_0(X_a) +1$. Moreover, the critical value $c$ corresponds to a minimal element of $X$.
\end{prop}

\begin{proof}
	Since $f$ is injective, then the number of path-components of $X_t$ can only increase by one when $t$ reaches a new regular or critical value. Denote by $c'$ the minimum value among the critical or regular values which are strictly greater than $a$ and such that $b_0(X_{c'})=b_0(X_a)+1$.  Let us denote the  element corresponding to the value $c'$ by $x$, i.e. $f(x)=c'$. We have to prove that $x$ is critical and that $x$ is minimal so $c=c'$ is the claimed value.
	
	Let us begin by proving that $x$ is minimal. Assume that $x$ is not minimal and we will arrive to a contradiction. Since we are adding a new path-component at $c'$, all the elements of $\partial x$ must satisfy that their values by $f$ are strictly greater than $f(x)=c'$. Moreover, by hypothesis, $\#\partial x\geq 2$, which is a contradiction with the fact that $f$ is a Morse function.
	
	Now, we prove that $c'$ must be a critical value. We argue by contradiction again. If $x$ is not a critical element, since $x$ is minimal, then there exists $y\succ x$ such that $f(y)<f(x)$. Therefore, $x\in X_{f(y)}$. Furthermore, applying that $\#\partial y\geq 2$ and the definition of Morse function to $y$, there exists $x'\prec y$ such that $f(y)>f(x')$. This is a contradiction with the fact that we are adding a new path-component when we reach $c'$.
\end{proof}

\begin{example}
	Consider the Morse function represented in Figure \ref{fig:ex:poset_prop:path_components_critical_values}. The value $3$ must correspond to a critical point since we are adding a new path-component ($b_0(X_3)=b_0(X_2)+1$). Moreover, the point corresponding to the value $3$ is of  zero height.
	\begin{figure}[htbp]
		\centering
		\includegraphics[scale=0.5]{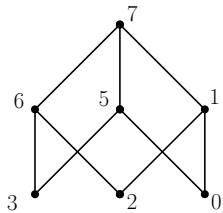}
		\caption{Morse function on a poset.}
		\label{fig:ex:poset_prop:path_components_critical_values}
	\end{figure}
\end{example}

With the following result we prove that the addition of regular elements does not create new connected components.

\begin{prop} \label{prop:regular_elements_do_not_create_path_components}
	Let $X$ be a path-connected  down-wide finite poset and let $f\colon X \to \mathbb{R}$ be an injective Morse function on it satisfying the Exclusion condition. Consider $a,b\in \R$, $a<b$.  If the interval $(a,b]$ does not contain critical values and only contains one regular value $f(y)=e$, then there exists $z\in X_a$ such that $z\prec y$ or $y\prec z$.
\end{prop}

\begin{proof}
	Due to Proposition \ref{prop:path_components_critical_values}, it follows that $b_0(X_a)\geq b_0(X_b)$, i.e, $y$ can not be in a different connected component.  Now the result follows from the fact that regular elements come in pairs.
\end{proof}

\begin{remark}
	Observe that the Exclusion Lemma (Lemma \ref{lma:exclusion:finite}) does not guarantee the conclusion of Proposition \ref{prop:regular_elements_do_not_create_path_components} by itself since we need to ensure that $y$ is not in a different connected component (as it happens with arbitrary posets, see Figure \ref{fig:counterexample_collapsing_theorem_regular_values} where $y$ is the element with the value $4$).
\end{remark}

The previous Propositions \ref{prop:path_components_critical_values} and \ref{prop:regular_elements_do_not_create_path_components}  may not hold for arbitrary posets, as the following example shows.

\begin{example}\label{ex:counterexample_collapsing_theorem_regular_values}
	Consider the Morse function  represented in Figure \ref{fig:counterexample_collapsing_theorem_regular_values}. The value $4$ is regular. However, $b_0(X_4)\neq b_0(X_3)$ while there are no critical values in $(3,4]$.
	\begin{figure}[htbp]
		\centering
		\includegraphics[scale=0.5]{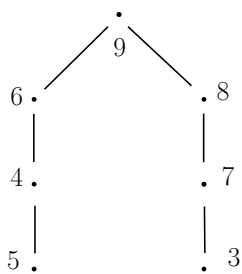}
		\caption{Regular values and path-components in general posets.}
		\label{fig:counterexample_collapsing_theorem_regular_values}
	\end{figure}
\end{example}

\subsection{Structural Theorems}

Both in smooth and discrete Morse Theory, manifolds and cell complexes can be recovered up to homotopy type from Morse functions defined on them by means of the so called fundamental theorems of Morse Theory. The next example shows that this is not possible in combinatorial Morse Theory defined on posets.

\begin{example}
	Consider the face poset of the simplicial complex depicted in Figure \ref{fig:counterexample_decomposition_homotopy_type}. It does not have the homotopy type of a point \cite[Example 5.1.12]{Barmak_book}. However, there is a Morse function defined on that poset with only one critical point, the Morse function associated to the matching drawn in the figure.
	\begin{figure}[htbp]
		\centering
		\includegraphics[width=4cm,height=3cm]{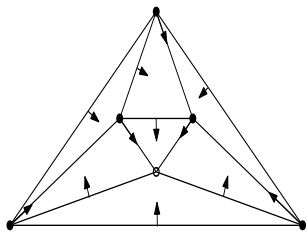}
		\caption{The Triangle.}
		\label{fig:counterexample_decomposition_homotopy_type}
	\end{figure}
\end{example}

This subsection is devoted to proving the substitutes of the fundamental theorems of Morse Theory in this context, that is, two collapsing theorems and an adjunction theorem. The first collapsing theorem is a homological collapsing theorem, which asserts that, in the absence of critical values, the homology remains unchanged provided the matching is homologically admissible. This result, combined with the adjunction theorem, is enough to prove the Morse inequalities. The second collapsing theorem guarantees that, in the absence of critical values, the weak homotopy type remains unchanged, provided that the matching is 1-admissible and homologically admissible. This result is analogous to  \cite[Theorem 3.3]{Forman2} in discrete Morse theory and plays the role of \cite[Theorem 3.1]{Milnor} in smooth Morse Theory. 

\begin{theorem}\label{thm:collapse_thm_homology} Let $X$ be a finite path-connected down-wide poset. Let $f\colon X \to \mathbb{R}$ be a Morse function satisfying the Exclusion condition.  Suppose that $(a,b]$ for $a<b$ contains no critical values of $f$. If  the Morse matching associated to the function $f$ is homologically admissible, then the inclusion $i \colon X_a \hookrightarrow X_b$ induces an isomorphism in homology.
\end{theorem}

In order to prove Theorem \ref{thm:collapse_thm_homology}  we begin with the following easy lemma which allow us to perturb the function locally so it can be taken to be injective.

\begin{lemma}\label{lema:injectivity_for_granted}
Let $X$ be a finite path-connected  down-wide  poset. Let $f\colon X \to \mathbb{R}$ be a Morse function satisfying the Exclusion condition.  Suppose that $(a,b]$ for $a<b$ contains no critical values of $f$ and contains at most one regular value $f(v)=c<b$.  Then there is a Morse function $g\colon X \to \mathbb{R}$ verifying: \begin{enumerate}
\item $X_t^g=X_t^f$ for every $t<c$. 
\item The restriction of $g$ to $(a,c]$ is injective.
\end{enumerate}
\end{lemma}

As a consequence of Lemma \ref{lema:injectivity_for_granted}, in the next proofs we will assume the injectivity of certain functions.

\begin{prop}\label{prop:trick}
	Let $X$ be a finite path-connected  down-wide poset. Let $f\colon X \to \mathbb{R}$ be a Morse function satisfying the Exclusion condition. Suppose that $(a,b]$ for $a<b$ contains no critical values of $f$ and contains at most one regular value $f(v)$.  Then $X_b= X_a$ or $X_b-X_a=\{v,w\}$, where $w\prec v$ and $w$ is an up beat point in $X_b$.
\end{prop}

\begin{proof} Assume that $f$ is injective and that furthermore, $(a,b]$ contains only one regular value $c=f(v)$, then $a<c<b$. Since $v$ is a regular point and $f$ satisfies the Exclusion condition, Proposition \ref{prop:regular_elements_do_not_create_path_components} implies that we only need to consider the following two mutually exclusive cases:
	\begin{enumerate}
		\item First, suppose that there exists $w\succ v$ such that $f(w)<f(v)$.  Now $f(w)\leq a$ since $f(v)$ is the unique regular value in $(a,b]$ and $f(w)<f(v)<b$. Therefore, $X_b=X_a$.
		\item For the second case, there exists $w\prec v$ with $f(w)>f(v)$. We claim that 	$v$ is the unique point in $X_b$ such that $w\prec v$ and $v\notin X_a$. That is, $X_b= X_a \sqcup \{v,w\}$ where $w\prec v$ and $w$ is an up beat point. In order to prove the claim, suppose there exists $u\neq v$, such that $w\prec u$.
		
		\begin{claim*}
			Under the above conditions, If $u\in X_b$, then $u\in X_a$.
		\end{claim*}
		\begin{proof}[Proof of the Claim]
			\begin{enumerate}
				\item 	First, since $f(v)<f(w)$ and $v$ is the unique regular element in $(a,b]$, we must have $b<f(w)<f(u)$ by the definition of a Morse function.
				\item Second, recall that $u \in X_b$ iff there exists $r\in X$, $u \leq r$ such that $f(r)\leq b$.
			\end{enumerate}
			Combining (a) and (b) it follows that if $u\in X_b$, then there exists $r\in X$, $r \neq u$, $u \leq r$ such that $f(r)\leq b$. Since $f(v)$ is the unique regular value in $(a,b]$ and $(a,b]$ contains no critical values of $f$, it follows that $f(r)\leq a$.
		\end{proof}
		So, suppose there exists $u\neq v$, such that $w\prec u$. By the Claim there are two cases to consider:
		\begin{enumerate}
			\item[1] $u \in X_b$. Then $u \in X_a$ and $w$ is a (up) beat point.
			\item[2] $u \notin X_b$ and, again, $w$ is a (up) beat point. \qedhere
		\end{enumerate}
	\end{enumerate} 
\end{proof}

\begin{remark}
	Theorem \ref{thm:collapse_thm} does not necessarily hold for arbitrary posets, as Example \ref{ex:counterexample_collapsing_theorem_regular_values} shows.
\end{remark}

\begin{prop}\label{prop:trick3}
	Under the conditions of Proposition \ref{prop:trick}, the inclusion $i \colon X_a \hookrightarrow X_b$ induces an isomorphism in all homology groups if and only if the Morse matching associated to the function $f$ is homologically admissible.
\end{prop}

\begin{proof}
By applying the Long Exact Sequence of homology to the pair $(X_b,X_a)$ it follows that $i \colon X_a \hookrightarrow X_b$ induces an isomorphism in all homology groups if and only if $H_*(X_b,X_a)\cong 0$.  As a consequence of Excision Theorem \cite[Theorem 2.20]{Hatcher},  given two open sets $A$ and $B$ which cover $X_b$, then there is an isomorphism $H_*(B,A\cap B)\cong H_*(X_b,A)$. Considering $A=X_a$ and $B=U_v$, it follows that
 $$H_*(U_v,\widehat{U}_v-\{w\})\cong H_*(X_b,X_a).$$
 Since $w\prec v$ is an element in the matching and the matching is homologically admissible, then $\widehat{U}_{v}-\{w\}$ is acyclic. By applying the Long Exact Sequence of homology to the pair $(U_v,\widehat{U}_v-\{w\})$ and using the fact that $U_v$ is contractible, it follows that $H_*(U_v,\widehat{U}_v-\{w\})\cong H_*(\widehat{U}_{v}-\{w\})$, so $H_*(U_v,\widehat{U}_v-\{w\})\cong 0$ if and only if the element of the matching $w\prec v$  is homologically admissible.
\end{proof}


\begin{proof}[Proof of Theorem \ref{thm:collapse_thm_homology}]
	It follows by combining Propositions \ref{prop:removing_beat_point_homot_equiv}, \ref{prop:trick} and \ref{prop:trick3}.
\end{proof}

Now we state the weak homotopical collapsing theorem. We need to add the extra hypothesis that the Morse matching associated to the function $f$  is $1$-admissible.

\begin{theorem}\label{thm:collapse_thm} Let $X$ be a finite path-connected down-wide poset. Let $f\colon X \to \mathbb{R}$ be a discrete Morse function satisfying the Exclusion condition.  Assume that $(a,b]$ for $a<b$ contains no critical values of $f$.
	\begin{enumerate}
		\item If the Morse matching associated to the function $f$ is $1$-admissible and homologically admissible, then the inclusion $i \colon X_a \hookrightarrow X_b$ is a weak homotopy equivalence.
		\item Moreover, in case $(a,b]$ contains no critical values of $f$ and contains at most one regular value $f(v)$, then $X_b= X_a$ or $X_b-X_a=\{v,w\}$ where $w\prec v$ and $w$ is an up beat point in $X_b$ and $v$ is a $\gamma-$point in $X_b-\{w\}$.
	\end{enumerate}
\end{theorem}

\begin{prop}\label{prop:trick2}
	Under the conditions of Proposition \ref{prop:trick}, the inclusion $i \colon (X_b-\{w\})-\{v\} \hookrightarrow X_b-\{w\}$ is a weak homotopy equivalence if and only if the element $w\prec v$ of the Morse matching associated to the function $f$ is $1$-admissible and homologically admissible. Moreover $v$ is a $\gamma-$point in $X_b-\{w\}$.
\end{prop}

\begin{proof}
	We will apply McCord's Theorem (Theorem \ref{thm:McCords_thm}) to the base $\{U_x\co x\in X_{b}-\{w\}\}$. There are two cases to consider:
	\begin{enumerate}
		\item If $x\neq v$, then $i^{-1}(U_{x})$ has a maximum and therefore is contractible, so $i_{\vert i^{-1}(U_{x})}\colon i^{-1}(U_{x})\to U_{x}$ is a weak homotopy equivalence.
		\item If $x= v$, then $i_{\vert i^{-1}(U_{x})}\colon i^{-1}(U_{x})\to U_{x}$ is the map $i\colon \widehat{U}_{v}-\{w\} \hookrightarrow U_{v}$. The subspace $U_{v}$ is contractible so it is homotopically trivial. Therefore $i\colon \widehat{U}_{v}-\{w\} \hookrightarrow U_{v}$ is a weak homotopy equivalence if and only if $\widehat{U}_{v}-\{w\}$ is homotopically trivial. Now, since $\widehat{U}_{v}-\{w\}$ is simply connected and acyclic, by Hurewicz Theorem it is homotopically trivial.
	\end{enumerate}
\end{proof}

At this point we can conclude:

\begin{proof}[Proof of Theorem \ref{thm:collapse_thm}]
	It follows by combining Propositions \ref{prop:removing_beat_point_homot_equiv}, \ref{prop:trick} and \ref{prop:trick2}.
\end{proof}

The following result explains what happens with the homotopy type when we reach critical values. It plays the role of \cite[Theorem 3.2]{Milnor} in the case of smooth Morse theory, and \cite[Theorem 3.4]{Forman2} in discrete Morse theory.

\begin{theorem}\label{thm:homot_poset_crit_points}
	Let $X$ be a path-connected down-wide finite poset and $f\colon X \to \R$ an order preserving Morse function satisfying the Exclusion condition. Suppose that $x^{(p)}$ is a critical point for $f$, that $f(x)\in (a,b]$ for $a<b$, and that there  are no other points in $f^{-1}((a,b])$. Then $X_b=X_a \cup_{\partial x^{(p)}}x^{(p)}$. 
\end{theorem}

\begin{proof}
	We may assume that $f$ is injective, that $f(x)>a$ and that the only point in $f^{-1}((a,b])$ is $x$. 
	Since $x$ is critical, then, given $y^{(p+1)}\succ x$, $f(y)>f(x)$. Hence, $f(y)>b$ and since $f$ is order preserving and satisfies the Exclusion condition, then  $f(z)>b$ for every $z>x$. Therefore, $x\cap X_a=\emptyset$. Given any $w^{(p-1)} \prec x$, due to the criticality of $x$, it holds that $f(w)<f(x)$. Therefore $f(w)\leq a$ and $w\in X_a$. Hence $\partial x \subseteq X_a$. That is, $X_b=X_a \cup_{\partial x^{(p)}}x^{(p)}$.
\end{proof}

\section{Consequences}\label{sec:consequences}

\subsection{Extension of Forman's Decomposition Theorem}
As a first consequence, we extend Forman's Discrete Morse theory on regular CW-complexes to more general cell complexes. We recover Forman's result \cite[Corollary 3.5]{Forman2} as a particular case. Moreover, we do not need to make use of simple homotopy types. We will work with less rigid cell structures than regular CW-complexes while maintaining some combinatorial structure. 

\begin{definition}\cite{Barmak_book,Minian}
	A CW-complex $K$ is h-regular if the attaching map of each cell is a homotopy equivalence onto its image and the closed cells are subcomplexes of $K$. Equivalently, the CW-complex $K$ is h-regular if the closed cells are contractible subcomplexes. 
\end{definition}

Given an h-regular CW-complex $K$, the cells whose attaching maps are not homeomorphisms are called {\em irregular} cells. For a detailed exposition of h-regular CW-complexes and some examples the reader is referred to \cite{Barmak_Minian,Barmak_book}. 

We recall Forman's definition of Morse function on an h-regular CW-complex $K$ \cite{Forman2}.

\begin{definition}
	Let $K$ be an h-regular CW-complex,  a  {\it discrete Morse function} on $K$ is a map $f\colon K \to \R$  such that, for every $p$-cell $\sigma^{(p)}\in K$, we have:\begin{enumerate}
		\item If $\sigma$ is an irregular face of $\tau^{(p+1)}$, then $f(\tau)>f(\sigma)$. Moreover, 
		$$
		\#\{\tau^{(p+1)}\in K\colon \sigma\prec \tau \text{ and } f(\sigma)\geq f(\tau)\}\leq 1.
		$$
		\item If $\beta^{(p-1)}$ is an irregular face of $\sigma$, then $f(\beta)<f(\sigma)$. Moreover, 
		$$
		\#\{\beta^{(p-1)}\in K \colon \beta\prec \sigma \text{ and } f(\beta)\geq f(\sigma)\}\leq 1.
		$$		 
	\end{enumerate}
\end{definition}

We present a generalized notion of a discrete Morse function on a h-regular CW-complex $K$:

\begin{definition}
	Let $K$ be a h-regular CW-complex,  a  {\it Morse function} on $K$ is a map $f\colon K \to \R$  such that, for every $p$-cell $\sigma^{(p)}\in K$, we have
	$$
	\#\{\tau^{(p+1)}\in K\colon \sigma\prec \tau \text{ and } f(\sigma)\geq f(\tau)\}\leq 1
	$$
	and
	$$
	\#\{\beta^{(p-1)}\in K \colon \beta\prec \sigma \text{ and } f(\beta)\geq f(\sigma)\}\leq 1.
	$$
\end{definition} 

For both definitions, a $p$-cell $\sigma$ is {\em critical of index} $p$ for $f\colon K \to \R$ if:
\begin{enumerate}
	\item $	\#\{\tau^{(p+1)}\in K\colon \sigma\prec \tau \text{ and } f(\sigma)\geq f(\tau)\}=0$ and
	\item $	\#\{\beta^{(p-1)}\in K \colon \beta\prec \sigma \text{ and } f(\beta)\geq f(\sigma)\}=0.	$
\end{enumerate}  

Our definition generalizes Forman's since we do not force non-regular cells to be critical.

We recall from the proof of \cite[Theorem 7.1.7]{Barmak_book} the following result, which is a generalization of Theorem \ref{thm:mccords_thms} (2): 

\begin{prop}\label{thm:weak_homot_equiv_h_reg_CW}
	Let $K$ be a finite h-regular CW-complex and let $L\subset K$ be a subcomplex.
	\begin{enumerate}
		\item There are maps $f_K\co K\to \Delta(K)$ and $f_L\co L\to \Delta(L)$ defined in \cite[Theorem 7.1.7]{Barmak_book} which are weak homotopy equivalences.
		\item The following diagram is commutative:
		$$\begin{tikzcd}
		L  \arrow[d, "f_L",] \arrow[r, hookrightarrow, " i "] & K \arrow[d, "f_K",]\\
		\Delta(L) \arrow[r, hookrightarrow, " i "]  &  \Delta(K).
		\end{tikzcd}$$  
		Therefore, $i\co L\hookrightarrow K$ is a (weak) homotopy equivalence if and only if $i\co \Delta(L)\hookrightarrow \Delta(K)$ is a weak homotopy equivalence.
	\end{enumerate}
\end{prop}

It is clear that for any CW-complex $K$, given a Morse function $f\colon K\to \R$, it induces a combinatorial Morse function on its face poset $\Delta(f)\colon \Delta(K)=X\to \R$ such that the face poset functor satisfies $\Delta(K_a)=\Delta(K)_a$.

\begin{corollary}\label{thm:decomposition_theorem_discrete_morse_theory}
	Let $K$ be a finite h-regular CW-complex and let $f\colon K\to \R$ be a Morse function  satisfying the Exclusion condition and such that the associated Morse matching in the face poset, $\mathcal{M}_{\Delta(f)}$, is $1$-admissible and homologically admissible. Then $K$ is homotopy equivalent to a CW-complex with one $p$-cell for each critical point of index $p$.
\end{corollary}

\begin{proof}
	First, consider the Morse function $\Delta(f)\colon \Delta(K)=X\to \R$ induced by $f\colon K\to \R$. By Corollary \ref{coro:exclusive_morse_function_implies_exclusive_order_preserving_one}, we can assume that $\Delta(f)\colon \Delta(K)=X\to \R$ is order preserving and satisfies the Exclusion condition.  Observe that the face poset of an h-regular CW-complex is a down-wide poset. Suppose that $(a,b]$ contains no critical values of $f$, $a<b$.  Then $i \colon X_a \hookrightarrow X_b$ is a weak homotopy equivalence by our collapsing theorem (Theorem \ref{thm:collapse_thm}). From Proposition \ref{thm:weak_homot_equiv_h_reg_CW}, $i \colon K_a \hookrightarrow K_b$ is a homotopy equivalence.   Suppose $x^{(p)}$ is a critical cell for $f$, with $f(x)\in (a,b]$, and that there  are no other cells in $f^{-1}((a,b])$. Then the same combinatorial proof of \cite[Theorem 3.4]{Forman2}  (observe that here we are using that $\Delta(f)\colon \Delta(K)=X\to \R$ is order preserving and satisfies the Exclusion condition) proves that $K_b$ has the same homotopy type as $K_a$ with a $p$-cell attached. The result follows.
\end{proof}	

\begin{remark}
	Corollary \ref{thm:decomposition_theorem_discrete_morse_theory} provides an alternative approach to \cite[Corollary 4.2]{Minian} by means of our Fundamental Theorems. Moreover, it does not involve the use of simple homotopy type.
\end{remark}

\subsection{Morse-Pitcher Inequalities}
Another consequence of our structural theorems of Morse Theory for finite spaces is that we can reproduce the classical proof (see \cite{Pitcher,Milnor} for the standard argument) of Morse-Pitcher inequalities in this context.

We consider coefficients in a principal ideal domain.  Let  $f\colon X\to \R$ be a Morse function, we denote by $m_i$ the number of critical points of height $i$ and by $b_i$ the Betti number of dimension $i$ without expliciting the space nor the domain of coefficients.

\begin{corollary}[Strong Morse inequalities]
	Let $X$ be a down-wide poset and let $f\colon X\to \R$ be an order preserving Morse function  satisfying the Exclusion condition. Suppose that the Morse matching associated to $f$ is homologically admissible and homology-regular. Then, for every $n\geq 0$ and domain of coefficients:
	$$m_n-m_{n-1}+\cdots +(-1)^n m_0\geq b_n-b_{n-1}+\cdots +(-1)^n b_0.$$
\end{corollary}

\begin{corollary}[Weak Morse inequalities]
	Let $X$ be a down-wide poset and let $f\colon X\to \R$ be an order preserving Morse function  satisfying the Exclusion condition. Suppose that the Morse matching associated to $f$ is homologically admissible and homology-regular. Then:
	\begin{enumerate}
		\item $m_i\geq b_i$ for every $i$.
		\item The Euler-Poincar\'e Characteristic satisfies $$\chi(X)=\sum_{i=0}^{\deg(X)}(-1)^i b_i=\sum_{i=0}^{\deg(X)}(-1)^i m_i.$$
	\end{enumerate}
\end{corollary}

\begin{remark}
	The Morse inequalities for homologically admissible posets can also be derived by following a combinatorial Hodge-theoretic argument mimicking \cite{Forman_Witten} since the arguments provided by Forman can be reproduced without changes in this context.
\end{remark}



Moreover, we also recover Pitcher strengthening of Morse inequalities. We denote by $\mu_i$ the minimum number of generators of the torsion subgroup $T_i$ of $H_i(X)$.

\begin{corollary}
	Let $X$ be a down-wide poset and let $f\colon X\to \R$ be an order preserving Morse function satisfying the Exclusion condition whose associate Morse matching is homologically admissible and homology-regular. Then it holds that: \begin{enumerate}
		\item For every $n\geq 0$:
		$$m_n\geq b_n+\mu_n+\mu_{n-1}.$$
		\item For every $n\geq 0$:
		$$m_n-m_{n-1}+\cdots +(-1)^n m_0\geq b_n-b_{n-1}+\cdots +(-1)^n b_0 + \mu_n.$$
	\end{enumerate}
\end{corollary}

Observe that if $\deg(X)=n$, then $\mu_n=0$ since $H_n(X)$ is a subgroup of the free abelian group $C_n(X)$. Moreover, $\mu_0=0$ and $\mu_{-1}$ is defined as $0$.

\subsection{Cancelling critical points}

Both the Morse and Morse-Pitcher inequalities suggest the study of Morse functions with few critical points, the so-called {\em optimal} Morse functions. In order to obtain such functions we present an approach consisting in canceling pairs of critical elements, extending to the context of posets known results on smooth manifolds and simplicial complexes.  First, we introduce some terminology. Given a matching $\mathcal{M}$ on the poset $X$, we will decompose $X$ as the disjoint union of three subsets $X=\crit(\mathcal{M})\sqcup s(\mathcal{M})\sqcup t(\mathcal{M})$. For each edge $(x,y)\in \mathcal{M}$, we say that $x$ is the {\em source} of the edge and $y$ is the {\em target}.  We define the {\em source of the matching} $s(\mathcal{M})$ as set whose elements are the sources of the edges in the matching. Analogously, we define the {\em target of the matching} $t(\mathcal{M})$ as set whose elements are the targets of the edges in the matching.  For convenience, we define the {\em source} and {\em target maps} (only defined for elements in the matching $\mathcal{M}$) as follows: given $(x,y)\in \mathcal{M}$, $s(y)=x$ and $t(x)=y$.

\begin{definition}
	Let $\mathcal{M}$ be a matching on the poset $X$. A $\mathcal{M}$-path of index $p$ from $x^{(p)}$ to $\tilde{x}^{(p)}$ is a sequence:
	$$\gamma \co x=x_0^{(p)}\prec y_0^{(p+1)} \succ x_1^{(p)}\prec y_1^{(p+1)} \succ \cdots \prec y_{r-1}^{(p+1)} \succ x_r^{(p)}=\tilde{x}$$
	such that for each $i\in \{0,\ldots, r-1\}$:
	\begin{enumerate}
		\item $(x_i,y_i) \in \mathcal{M}$,
		\item $x_i\neq x_{i+1}$.
	\end{enumerate}
\end{definition}

We present a result, which can be seen as the adaptation of \cite[Theorem 11.1]{Forman2} to our context.

\begin{theorem}[Canceling critical points] \label{thm:cancelling_critical_points} 
	Given a matching $\mathcal{M}$ on a finite graded poset $X$, assume that $z^{(p+1)}$ and $x^{(p)}$ are critical points such that there is an element $y^{(p)}\prec z^{(p+1)}$ and an unique $\mathcal{M}$-path
	$$\gamma \co z\succ y=x_0\prec z_0 \succ x_1 \prec z_1 \succ \cdots \prec z_r\succ x_r=x$$
	(there is no other $\mathcal{M}$-path from any $p$-face of $z^{(p+1)}$ to $x^{(p)}$). Then there is a matching $\mathcal{M}'$ such that:
	\begin{itemize}
		\item The set of critical points of $\mathcal{M}'$ is $$\crit(\mathcal{M}')=\crit(\mathcal{M})-\{x,z\}.$$
		\item Moreover, $\mathcal{M}'=\mathcal{M}$ except along the unique gradient path from $\partial z$ to $x$.
	\end{itemize}
	
\end{theorem}

\begin{proof}
	We define $\mathcal{M}'$ as follows:
	\begin{enumerate}
		\item $t_{\mathcal{M}'}(w)=t_{\mathcal{M}}(w)$ if  $w \notin \{y,z_0,x_1,z_1,\ldots z_r,x\}$ ($\mathcal{M}'=\mathcal{M}$ except along the unique gradient path from $\partial z$ to $x$)
		\item $t_{\mathcal{M}'}(x_i)=z_{i-1}$, $i=1,\ldots,r$ (we reverse the gradient path from $x$ to $z_0$ so $x$ is no longer critical)
		\item $t_{\mathcal{M}'}(y)=z$ (we reverse the arrow from $y$ to $z$ so $z$ is no longer critical).
	\end{enumerate}
	It remains to check that there are no closed $\mathcal{M}'$-paths. We argue by contradiction. Suppose there was a closed $\mathcal{M}'$-path $\delta$.

	\begin{claim*}
		Under the above hypothesis, $\delta$ would contain at least one $p$-element from $\gamma$ and one $p$-element not in $\gamma$.
	\end{claim*}
	\begin{proof}[Proof of the Claim]
		The elements coming from $\gamma$ can not give a closed $\mathcal{M}'$-path on their own since we have just reverted their arrows. The elements of $X$ which are not in $\gamma$ can not give a closed $\mathcal{M}'$-path since in that case we would also have a closed $\mathcal{M}$-path and $\mathcal{M}$ is a gradient vector field. Therefore in $\delta$ we must have at least one element in each of their sets. Moreover, if we have a $(p+1)$ element, then we have a $p$-element, so we have at least one $p$-element of each of one of these sets.
	\end{proof}

	Hence, $\delta$ would contain a sequence of the form:
	$$x_i\prec w_0\succ s_1\prec w_1 \succ \cdots \prec w_s\succ x_j$$
	where $s\geq 0$, $w_i\neq x_k$, $w_i\neq z_k$, $s_i\neq x_k$, $s_i\neq z_k$, for all $i$ and $k$.
	Since $t_{\mathcal{M}'}(w_i)=t_{\mathcal{M}}(w_i)$ and $t_{\mathcal{M}'}(s_i)=t_{\mathcal{M}}(s_i)$ for all $i$, we have a $\mathcal{M}$-path:
	$$w_0 \succ s_1 \prec w_1 \succ \cdots \prec w_s \succ x_j.$$
	Let us consider two cases:
	\begin{enumerate}
		\item If $i\neq 0$, then $s_1\neq x_{i-1},x_i$ and $s_1\prec t_{\mathcal{M}'}(x_i)=z_{i-1}$. Therefore, we can define a second gradient $\mathcal{M}$-path $\gamma'\neq \gamma$ from $\partial z$ to $x$:
		\begin{align*}
			\gamma'\co y=&x_0\prec z_0 \succ x_1 \prec \cdots \succ x_{i-1} \prec z_{i-1} \succ s_1\prec w_1 \succ \cdots\\
			&\succ x_j \prec z_j\succ \cdots \succ x_r=x.
		\end{align*}	
		
		Which is a contradiction.
		\item If $i=0$, then $y=x_0\neq s_1 \prec t_{\mathcal{M}'}(y)=z$. Therefore, we can define the following $\mathcal{M}$-path:
		$$\gamma' \co z \succ s_1 \prec w_1 \succ \cdots \prec w_s \succ x_j \prec z_j \succ \cdots \succ x_r=x$$
		which is different from $\gamma$ and also goes from $\partial z$ to $x$. Then we have a contradiction. \qedhere
	\end{enumerate}
\end{proof}

Finally, there is a kind of dual result to Theorem \ref{thm:cancelling_critical_points} which allows us to create critical points. Both the statement and the proof are a straightforward translation of \cite[Theorem 11.3]{Forman2}.

\bibliographystyle{abbrv}
\bibliography{biblio}

\end{document}